\documentclass[suppldata]{interact}

\usepackage{epstopdf} 
\usepackage{amsmath}
\usepackage[caption=false]{subfig} 
\usepackage{amsfonts}

\usepackage[longnamesfirst,sort]{natbib} 
\bibpunct[, ]{(}{)}{;}{a}{,}{,} 
\renewcommand\bibfont{\fontsize{10}{12}\selectfont}

\usepackage{hyperref}
\usepackage{bbm}

\theoremstyle{plain}
\newtheorem{theorem}{Theorem}[section]
\newtheorem{lemma}[theorem]{Lemma}

\newtheorem{mydef*}{Definition*}

\theoremstyle{definition}
\newtheorem{definition}[theorem]{Definition}

\theoremstyle{remark}

\begin{document}
\title{SOCLE OF A HAMILTONIAN GROUP}
\author{
\name{Sourav Koner and Biswajit Mitra}\thanks{CONTACT Email:  harakrishnaranusourav@gmail.com \& bmitra@math.buruniv.ac.in}
\affil{Department of Mathematics, The University of Burdwan, India 713104 \& Department of Mathematics, The University of Burdwan, India 713104}}

\maketitle
\begin{abstract}
The socle of a group $G$ is the subgroup generated by all minimal normal subgroups of $G$. In this short note, we determine the socle of a Hamiltonian group explicitly.
\end{abstract}
 
\section{Introduction}

A group is called a Dedekind group if all of its subgroups are normal (see [\cite{RD97}]). Therefore, all abelian groups are Dedekind groups. A non-abelian Dedekind group is called a Hamiltonian group. The smallest example of Hamiltonian group is $Q_{8}$, the quaternion group of order $8$. In both finite and infinite order cases, Dedekind and Baer have shown that every Hamiltonian group is a direct product of $Q_{8}$, a group of exponent $2$, and a torsion abelian group in which all elements have odd order (see [\cite{BR33}], [\cite{HM99}]). In particular, Hamiltonian groups are locally finite, nilpotent of class $2$, and solvable of length $2$. In the present paper we prove the following theorem:

\begin{theorem}\label{srmkrsna}
If $\mathbbm{H} = Q_{8} \oplus B \oplus D$, where $B$ is an elementary abelian $2$-group, and $D$ is a torsion abelian group with all elements of odd order, then $\mathrm{Soc}(\mathbbm{H}) \simeq \mathbb{Z}_{2} \oplus B \oplus \mathbbm{P}(D)$, where $\mathbbm{P}(D)$ is the group generated by all the elements of $D$ that have prime orders.
\end{theorem}

For a finite solvable group $G$, the socle is a product of elementrary abelian $p$-groups for a collection of primes dividing the order of $G$. However, this may not include all primes that divide the order of $G$. This result follows from a well-known fact that the socle of a finite nilpotent group is a product of elementary abelian $p$-groups for the collection of primes dividing the order of the group. The main objective of above stated theorem is that it explicitly determines the socle of a Hamiltonian group. We prove theorem \eqref{srmkrsna} shortly, before that we develop a few facts about proper essential subgroups (see [\cite{PG70}]).

\begin{definition}\label{madurga}
A proper nontrivial subgroup $E$ of a group $G$ is said to be proper essential in $G$ if $E \cap H \neq \{1_{G}\}$ for every nontrivial subgroup $H$ of $G$. \end{definition}

Hereafter, we write \emph{$E$ is a proper essential subgroup of $G$} or \emph{$E$ is proper essential in $G$} if $E$ satisfies definition \eqref{madurga}.

\section{Proper essential subgroups}

\begin{theorem}\label{maa}
Let $\{G_{\omega}\}_{\omega \in \Lambda}$ be a family of groups indexed by a nonempty set $\Lambda$. Then $\bigoplus_{\omega \in \Lambda} G_{\omega}$ has a proper essential subgroup if and only if some $G_{\omega}$ has a proper essential subgroup. 
\end{theorem}
\begin{proof}
Let $G = \bigoplus_{\omega \in \Lambda} G_{\omega}$ and $\omega_{0} \in \Lambda$ be such that $G_{\omega_{0}}$ has a proper essential subgroup $E_{\omega_{0}}$. Let $X = \bigoplus_{\omega \in \Lambda} X_{\omega}$, where $X_{\omega_{0}} = E_{\omega_{0}}$ and $X_{\omega} = G_{\omega}$ if $\omega \neq \omega_{0}$. We claim that $X$ is a proper essential subgroup of $G$. Let $N$ be any nontrivial subgroup of $G$ and let $(n_{\omega})_{\omega \in \Lambda}$ be a nonzero element of $N$. Clearly, $(n_{\omega})_{\omega \in \Lambda} \in X$ if $n_{\omega_{0}} = 1_{G_{\omega_{0}}}$. If $n_{\omega_{0}} \neq 1_{G_{\omega_{0}}}$, consider the cyclic group $\langle n_{\omega_{0}} \rangle$. As $E_{\omega_{0}} \cap \langle n_{\omega_{0}} \rangle \neq \{1_{G_{\omega_{0}}}\}$, we conclude that $X$ is a proper essential subgroup of $G$. 

For the converse, assume that $G$ has a proper essential subgroup $E$ but $G_{\omega}$ does not have a proper essential subgroup for all $\omega \in \Lambda$. Observe that, for all $\omega \in \Lambda$, we have either $E \cap G_{\omega}$ is proper essential in $G_{\omega}$ or $G_{\omega} \subseteq E$. As $G_{\omega}$ does not have any proper essential subgroup, it must be that $G_{\omega} \subseteq E$ for all $\omega \in \Lambda$, that is, $G \subseteq E$, a contradiction. \end{proof}

\begin{definition}
Let $G$ be a group. If $G$ has a proper essential subgroup, we define $\delta(G)$ to be the intersection of all proper essential subgroup of $G$, and $G$ otherwise.  
\end{definition}

Observe that $\delta(G)$ is a characteristic subgroup of $G$ and hence is normal in $G$. Also, we have $\delta(G) = G$ if $G$ is a finite simple group, because, if $G$ has a proper essential subgroup, then $\delta(G)$ is a proper essential, therefore, proper non-trivial characteristic subgroup  of $G$, a contradiction. Now, if $G$ is a finite group such that $\mathrm{Soc}(G) = G$, then $G = \bigoplus_{\alpha = 1}^{n} S_{\alpha}$, where each $S_{\alpha}$ is a simple group for $1 \leq \alpha \leq n$. Thus, applying theorem \eqref{maa} we can conclude that $G$ does not contain proper essential subgroup. In the following theorem, we establish that $\delta$ commutes with the direct sum. 

\begin{theorem}\label{ons}
If $\{G_{\omega}\}_{\omega \in \Lambda}$ be a family of groups indexed by a nonempty set $\Lambda$, then $\delta\big(\bigoplus_{\omega \in \Lambda} G_{\omega}\big) = \bigoplus_{\omega \in \Lambda} \delta(G_{\omega})$.
\end{theorem}
\begin{proof}
Let $G = \bigoplus_{\omega \in \Lambda} G_{\omega}$, $G^{\prime} = \bigoplus_{\omega \in \Lambda} \delta(G_{\omega})$, $\mathcal{T} = \{\omega \in \Lambda \mid \delta(G_{\omega}) = G_{\omega}\}$ and $\mathcal{T^{\prime}} = \{\omega \in \Lambda \mid \delta(G_{\omega}) \neq G_{\omega}\}$. Clearly the sets $\mathcal{T}$ and $\mathcal{T^{\prime}}$ are disjoint and $\mathcal{T} \cup \mathcal{T^{\prime}} = \Lambda$. We have the following cases: $(a)$ $\mathcal{T} \neq \emptyset$, $\mathcal{T^{\prime}} = \emptyset$; $(b)$ $\mathcal{T} = \emptyset$, $\mathcal{T^{\prime}} \neq \emptyset$; $(c)$ $\mathcal{T} \neq \emptyset$, $\mathcal{T^{\prime}} \neq \emptyset$. 

$(a):$ Since $G$ does not have any proper essential subgroup (see theorem \eqref{maa}), we conclude that $\delta(G) = G^{\prime}$.

$(b):$ Let $(g_{\omega})_{\omega \in \Lambda} \in \delta(G)$, $\tau \in \Lambda$ be any element and $E_{\tau}$ be any proper essential subgroup of $G_{\tau}$. Consider $X = \bigoplus_{\omega \in \Lambda} X_{\omega}$, where $X_{\omega} = G_{\omega}$ for $\omega \neq \tau$ and $X_{\omega} = E_{\tau}$ for $\omega = \tau$. Since $X$ is proper essential
in $G$, we have $g_{\tau} \in E_{\tau}$. As $E_{\tau}$ was arbitrary we get $g_{\tau} \in \delta(G_{\tau})$. But $\tau$ was arbitrarily chosen as well. Hence, $(g_{\omega})_{\omega \in \Lambda} \in G^{\prime}$, that is, $\delta(G) \subseteq G^{\prime}$. Now, let $E$ be any proper essential subgroup of $G$. Observe that for any $\omega \in \Lambda$ we have either $E \cap G_{\omega}$ is proper essential in $G_{\omega}$ or $G_{\omega} \subseteq E$. But either of the cases give us $\delta(G_{\omega}) \subseteq E$. As $\omega$ and $E$ were arbitrary, we get that $G^{\prime} \subseteq \delta(G)$. 

$(c):$ Let $(g_{\omega})_{\omega \in \Lambda} \in \delta(G)$, $\tau \in \mathcal{T^{\prime}}$ and $E_{\tau}$ be any proper essential subgroup of $G_{\tau}$. Consider $X = \bigoplus_{\omega \in \Lambda} X_{\omega}$, where $X_{\omega} = G_{\omega}$ for $\omega \in \Lambda \setminus \{\tau\}$ and $X_{\tau} = E_{\tau}$. Since $X$ is proper essential in $G$, we have $g_{\tau} \in E_{\tau}$. As $\tau$ was arbitrary in $\mathcal{T^{\prime}}$, we get that $g_{\tau} \in \delta(G_{\tau})$ for all $\tau \in \mathcal{T^{\prime}}$, that is, $\delta(G) \subseteq G^{\prime}$. Now, let $E$ be any proper essential subgroup of $G$. Observe that for any $\omega \in \mathcal{T^{\prime}}$ we have either $E \cap G_{\omega}$ is proper essential in $G_{\omega}$ or $G_{\omega} \subseteq E$. But either of the cases give us $\delta(G_{\omega}) \subseteq E$. As $\omega$ was arbitrary in $\mathcal{T^{\prime}}$ and $G_{\lambda} \subseteq E$ for all $\lambda \in \mathcal{T}$, we get that $G^{\prime} \subseteq \delta(G)$.  
\end{proof}

\begin{lemma}\label{spmkj}
If a torsion group $G$ has a proper essential subgroup, then $\delta(G)$ is proper essential in $G$. Moreover, we have $\delta(G) = \mathbbm{P}(G)$, where $\mathbbm{P}(G)$ is the group generated by all the elements of $G$ that have prime orders.
\end{lemma}
\begin{proof}
Assume on the contrary that $\delta(G)$ is not proper essential in $G$. Therefore, there exists a subgroup $H \neq \{1_{G}\}$ of $G$ such that $H \cap \delta(G) = \{1_{G}\}$. If $h \in H \setminus \{1_{G}\}$ is such that $|h| = n$, then $\langle h \rangle \cap \delta(G) = \{1_{G}\}$. Thus, for each $i$ with $1 \leq i \leq n-1$, there exists at least one proper essential subgroup $E_{i}$ of $G$ such that $h^{i} \in E_{i}^{c}$. But this shows that $\langle h \rangle \setminus \{1_{G}\} \subseteq \bigcup_{i=1}^{n-1} E_{i}^{c}$, that is, $\big(\bigcap_{i=1}^{n-1}E_{i}\big) \bigcap \langle h \rangle = \{1_{G}\}$. As finite intersection of proper essential subgroups of $G$ is again proper essential in $G$, we get that $\bigcap_{i=1}^{n-1} E_{i}$ is a proper essential subgroup of $G$ which intersects $\langle h \rangle$ trivially, a contradiction. 

Finally, if $g \in G$ be such that $g^{p} = 1_{G}$ for some prime $p$, then $g \in \delta(G)$. This shows that $\mathbbm{P}(G) \subseteq \delta(G)$. Now, as $\mathbbm{P}(G)$ is also proper essential in $G$ and $\delta(G)$ is the intersection of all proper essential subgroups of $G$, we get that $\delta(G) \subseteq \mathbbm{P}(G)$.
\end{proof}

We now show that for a Hamiltonian group $\mathbbm{H}$, $\mathrm{Soc}(\mathbbm{H})$ is the intersection of all proper essential subgroups of $\mathbbm{H}$.  

\begin{theorem}\label{debanandamj}
If $\mathbbm{H}$ is a Hamiltonian group, then $\delta(\mathbbm{H}) = \mathrm{Soc}(\mathbbm{H})$.
\end{theorem}
\begin{proof}
Let $\mathbbm{H} = Q_{8} \oplus B \oplus D$, where $B$ is an elementary abelian $2$-group, and $D$ is a torsion abelian group with all elements of odd order. Now, applying theorem \eqref{maa} we get that $\mathbbm{H}$ has a proper essential subgroup. Since $\mathbbm{H}$ is a torsion group, applying lemma \eqref{spmkj} we conclude that $\delta(\mathbbm{H})$ is a proper essential subgroup of $\mathbbm{H}$. If $N$ is any minimal normal subgroup of $\mathbbm{H}$, then $N \cap \delta(\mathbbm{H}) \neq \{1_{\mathbbm{H}}\}$ shows that $N \subseteq \delta(\mathbbm{H})$. As $N$ was arbitrary, we conclude that $\mathrm{Soc}(\mathbbm{H}) \subseteq \delta(\mathbbm{H})$. As any nontrivial subgroup $K$ of $\mathbbm{H}$ contains a minimal normal subgroup, we conclude that $\mathrm{Soc}(\mathbbm{H})$ is a proper essential subgroup of $\mathbbm{H}$. Hence, we get that $\delta(\mathbbm{H}) \subseteq \mathrm{Soc}(\mathbbm{H})$.
\end{proof}

\section{Proof of theorem 1.1.}

Applying theorem \eqref{ons} and \eqref{debanandamj}, we conclude that $\mathrm{Soc}(\mathbbm{H}) = \delta(Q_{8}) \oplus \delta(B) \oplus \delta(D)$. We claim that $\delta(B) = B$. Because, if $B$ possess a proper essential subgroup $E$, then, as all elements of $B$ have order $2$, we must have $b \in E$ for all $b \in B$. But this implies $E = B$. This shows that $B$ does not contain any proper essential subgroup. Since, we have $\delta(Q_{8}) \simeq \mathbb{Z}_{2}$, so, applying lemma \eqref{spmkj} we conclude that $\mathrm{Soc}(\mathbbm{H}) \simeq \mathbb{Z}_{2} \oplus B \oplus \mathbbm{P}(D)$.

\end{document}